\documentclass[a4paper,12pt]{article}
\usepackage{amsmath,amsthm,amssymb,amscd}
\usepackage{ascmac}
\usepackage[arrow,matrix]{xy}
\usepackage{enumerate}
\usepackage{graphicx}

\theoremstyle{definition}
\newtheorem{thm}{Theorem}[section]

\newtheorem{prop}[thm]{Proposition}
\newtheorem{cor}[thm]{Corollary}
\newtheorem{lem}[thm]{Lemma}

\newtheorem{rem}[thm]{Remark}

\def\hlam{\hat{\lambda}}
\def\hmu{\hat{\mu}}
\def\hnu{\hat{\nu}}
\def\hal{\hat{\alpha}}
\def\hbeta{\hat{\beta}}
\def\hgam{\hat{\gamma}}
\def\gr{\mathop{\mathrm{gr}}\nolimits}
\def\bm{\bar{m}}

\newcommand{\mf}[1]{{\mathfrak{#1}}}

\newcommand{\bb}[1]{{\mathbb{#1}}}
\newcommand{\mca}[1]{{\mathcal{#1}}}
\newcommand{\mr}[1]{{\mathrm{#1}}}
\newcommand{\relmiddle}[1]{\mathrel{}\middle#1\mathrel{}}

\title{An algebro-geometric realization of the cohomology ring of Hilbert scheme of points in the affine plane}
\author{Tatsuyuki Hikita\footnote{thikita@math.kyoto-u.ac.jp}}
\date{}
\allowdisplaybreaks

\begin{document}

\maketitle

\begin{abstract}
We show that the cohomology ring of Hilbert scheme of $n$-points in the affine plane is isomorphic to the coordinate ring of $\bb{G}_{m}$-fixed point scheme of the $n$-th symmetric product of $\bb{C}^{2}$ for a natural $\bb{G}_{m}$-action on it. This result can be seen as an analogue of a theorem of DeConcini, Procesi and Tanisaki on a description of the cohomology ring of Springer fiber of type A.
\end{abstract}

\section{Introduction}

In \cite{DP} and \cite{T}, DeConcini-Procesi and Tanisaki show that the cohomology ring of Springer fiber of type A is isomorphic to the coordinate ring of scheme-theoretic intersection of some nilpotent orbit closure and Cartan subalgebra. The purpose of this paper is to generalize their result to wider situations by reinterpreting them as a certain isomorphism between the cohomology ring of some symplectic variety and the coordinate ring of some scheme coming from another symplectic variety.

First we recall the result of DeConcini-Procesi and Tanisaki. Let $G=\mr{GL}_{n}(\bb{C})$. We fix a Borel subgroup $B\subset G$ and a Cartan subgroup $T\subset B$. We take a parabolic subgroup $B\subset P\subset G$ and its Levi subgroup $L$. We denote by $\mf{g}$, $\mf{b}$, $\mf{t}$, $\mf{p}$, and $\mf{l}$ the Lie algebras of $G$, $B$, $T$, $P$, and $L$ respectively. Let $\mf{n}$ and $\mf{n}_{P}$ be the nilpotent radicals of $\mf{b}$ and $\mf{p}$. Let $\mca{N}\subset\mf{g}$ be the nilpotent cone of $\mf{g}$ and let $\mca{N}_{P}=\mr{Ad}(G)\cdot\mf{n}_{P}\subset\mca{N}$ be a closed subvariety of $\mca{N}$. If $\lambda\vdash n$ is the partition of $n$ corresponding to $P$, then $\mca{N}_{P}$ is the closure of the nilpotent orbit whose Jordan block is of type $\lambda^{T}$. Here, $\lambda^{T}$ is the transpose of $\lambda$.

We take a regular nilpotent element $e$ in $\mf{l}$. Consider the Springer resolution $$T^{\ast}(G/B)\cong\{(gB,X)\in G/B\times\mf{g}\mid \mr{Ad}(g)^{-1}(X)\in\mf{n}\}\xrightarrow{\mu}\mca{N}$$ given by $\mu(gB,X)=X$. Let $\mca{B}_{e}:=\mu^{-1}(e)$ be the Springer fiber associated with $e$. 

\begin{thm}[\cite{DP},\cite{T}]\label{DPT}
There is an isomorphism of graded algebras $$H^{\ast}(\mca{B}_{e},\bb{C})\cong\bb{C}[\mca{N}_{P}\cap\mf{t}].$$
Here, $\mca{N}_{P}\cap\mf{t}$ is the scheme-theoretic intersection of $\mca{N}_{P}$ and $\mf{t}$ in $\mf{g}$ and the grading on $\bb{C}[\mca{N}_{P}\cap\mf{t}]$ comes from the $\bb{G}_{m}$-action on $\mca{N}_{P}\cap\mf{t}$ induced by the scaling action $t\cdot X=t^{-2}X$ for $t\in\bb{G}_{m}$ and $X\in\mf{g}$. 
\end{thm}

We summarize the above theorem in the following diagram:
$$
\begin{array}{cccc}
T^{\ast}(G/P)&\hookleftarrow &G/P=\mu_{P}^{-1}(0)&\\
\hspace{0.5em}\downarrow^{\mu_{P}}&&&\\
\mca{N}_{P}&\hookleftarrow&\mca{N}_{P}\cap\mf{t}&\cong\mr{Spec}H^{\ast}(\mca{B}_{e},\bb{C}).\\
\end{array}
$$
Here, $\mu_{P}$ is the parabolic analogue of Springer resolution $$T^{\ast}(G/P)\cong\{(gP,X)\in G/P\times\mf{g}\mid \mr{Ad}(g)^{-1}(X)\in\mf{n}_{P}\}\xrightarrow{\mu_{P}}\mca{N}_{P}$$ given by $\mu(gP,X)=X$. For type $A$, $\mu_{P}$ always gives an resolution of singularities of $\mca{N}_{P}$ and $\mca{N}_{P}$ is normal. Hence $\mca{N}_{P}$ can be understood as the affinization of $T^{\ast}(G/P)$. We also note that $T^{\ast}(G/P)$ is homotopy equivalent to $G/P$.

On the other hand, we will give an algebro-geometric realization of $H^{\ast}(G/P,\bb{C})$ in the appendix of this paper. Let us take a $\mf{sl}_{2}$-triple $\{e,h,f\}$ containing $e$. Let $Z_{\mf{g}}(f)$ (resp. $Z_{\mf{l}}(f)$) be the centralizer of $f$ in $\mf{g}$ (resp. $\mf{l}$). Consider the Slodowy slice $S_{e}:=\mca{N}\cap(e+Z_{\mf{g}}(f))$. There is a $\bb{G}_{m}$-action on $S_{e}$ given by $t\cdot X=t^{-2}\mr{Ad}(t^{h})X$ for $t\in\bb{G}_{m}$ and $X\in S_{e}$. 

\begin{prop}[Theorem A.1 for $P=B$,$Q=P$]
There is a graded algebra isomorphism $$H^{\ast}(G/P,\bb{C})\cong\bb{C}[S_{e}\cap(e+Z_{\mf{l}}(f))].$$ Here, $S_{e}\cap(e+Z_{\mf{l}}(f))$ is the scheme-theoretic intersection of $S_{e}$ and $e+Z_{\mf{l}}(f)$ in $\mf{g}$ and the grading on $\bb{C}[S_{e}\cap(e+Z_{\mf{l}}(f))]$ comes from the $\bb{G}_{m}$-action on $S_{e}$ above. 
\end{prop}

We summarize this proposition in the following diagram:
$$
\begin{array}{cccc}
\tilde{S}_{e}&\hookleftarrow&\mca{B}_{e}=\mu_{e}^{-1}(e)&\\
\hspace{0.5em}\downarrow^{\mu_{e}}&&&\\
S_{e}&\hookleftarrow&S_{e}\cap(e+Z_{\mf{l}}(f))&\cong\mr{Spec}H^{\ast}(G/P,\bb{C}).\\
\end{array}
$$
Here, $\tilde{S}_{e}:=\mu^{-1}(S_{e})$ is the Slodowy variety and $\mu_{e}$ is the restriction of $\mu$ to $\tilde{S}_{e}$. It is known that $\mu_{e}$ gives a resolution of singularities of $S_{e}$ and $S_{e}$ is the affinization of $\tilde{S}_{e}$. Moreover, $\tilde{S}_{e}$ is homotopy equivalent to $\mca{B}_{e}$.  

Note that there is some similarity between the above two diagrams, where the roles of $\mca{B}_{e}$ and $G/P$ are exchanged to each other. It is known that $T^{\ast}(G/P)$ and $\tilde{S}_{e}$ are related to each other by symplectic duality in the sense of Braden, Licata, Proudfoot, and Webster (\cite{BLPW}). The aim of this paper is to generalize the above theorem of DeConcini-Procesi and Tanisaki to other cases of symplectic duality. For this purpose, we have to understand the scheme-theoretic intersections $\mca{N}_{P}\cap\mf{t}$ and $S_{e}\cap(e+Z_{\mf{l}}(f))$ more intrinsically.

Let us recall the notion of fixed point scheme (see \cite{F}). Let $H$ be an algebraic group over $\bb{C}$ and let $X$ be a scheme over $\bb{C}$ with $H$-action. Consider the contravariant functor $h^{H}_{X}$ from the category of $\bb{C}$-schemes to the category of sets given by $$h^{H}_{X}(Y)=(\mbox{the set of }H\mbox{-equivariant morphisms }Y\rightarrow X),$$ where $Y$ is a $\bb{C}$-scheme equipped with trivial $H$-action. This functor is known to be representable by a closed subscheme $X^{H}$ of $X$, which is called $H$-fixed point scheme of $X$. For $X=\mr{Spec}(A)$, the ideal of definition of $H$-fixed point scheme $X^{H}$ in $X$ is generated by all $h\cdot f-f$ for $h\in H(\bb{C})$ and $f\in A$.  

Let us consider the adjoint action of $T$ on $\mca{N}_{P}$. One can easily see that the scheme-theoretic intersection $\mca{N}_{P}\cap\mf{t}$ is isomorphic to the $T$-fixed point scheme $\mca{N}_{P}^{T}$ of $\mca{N}_{P}$ (or $\bb{G}_{m}$-fixed point scheme for some generic subgroup $\bb{G}_{m}\subset T$). Similarly, the scheme-theoretic intersection $S_{e}\cap(e+Z_{\mf{l}}(f))$ can be understood as the $Z(L)$-fixed point scheme $S_{e}^{Z(L)}$ of $S_{e}$ for the adjoint action of the center $Z(L)$ of $L$ on $S_{e}$.

Therefore, the above results can be considered as examples of the phenomenon that the cohomology ring of a conical symplectic resolution is isomorphic to the coordinate ring of a $\bb{G}_{m}$-fixed point scheme of the affinization of symplectic dual conical symplectic resolution. Main result of this paper show that this phenomenon occurs for Hilbert scheme of points in the affine plane. 

Let us explain the main result of this paper. Let $\mr{Hilb}^{n}(\bb{C}^{2})$ be Hilbert scheme of $n$-points in the affine plane (see \cite{N}). The affinization of $\mr{Hilb}^{n}(\bb{C}^{2})$ is given by $n$-th symmetric product $S^{n}\bb{C}^{2}$ of $\bb{C}^{2}$ and the Hilbert-Chow morphism $\mr{Hilb}^{n}(\bb{C}^{2})\rightarrow S^{n}\bb{C}^{2}$ gives a resolution of singularities. Let us consider the action of $\bb{T}=\bb{G}_{m}$ on $S^{n}\bb{C}^{2}$ induced by its action on $\bb{C}^{2}$ given by $t\cdot (x,y)=(t^{-1}x,ty)$ for $t\in\bb{T}$ and $(x,y)\in\bb{C}^{2}$. Since $\mr{Hilb}^{n}(\bb{C}^{2})$ is symplectic dual to itself (\cite{BLPW}), we come to the following statement by applying the above consideration to this case:

\begin{thm}\label{MainThm}
There is an isomorphism of graded algebras $$H^{\ast}(\mr{Hilb}^{n}(\bb{C}^{2}),\bb{C})\cong\bb{C}[(S^{n}\bb{C}^{2})^{\bb{T}}].$$
Here, the grading on $\bb{C}[(S^{n}\bb{C}^{2})^{\bb{T}}]$ comes from the $\bb{G}_{m}$-action induced by the its action on $\bb{C}^{2}$ given by $s\cdot (x,y)=(s^{-1}x,s^{-1}y)$ for $s\in\bb{G}_{m}$ and $(x,y)\in\bb{C}^{2}$.
\end{thm}

The rest of the paper is devoted to the proof of this theorem. We also prove in the appendices that the above phenomenon also occurs for the case of S3-varieties or hypertoric varieties. It would be interesting to find some conditions under which this kind of phenomenon can be expected to hold for more general symplectic dual pair of conical symplectic resolutions.

\subsection*{Acknowledgement}
The author thanks Syu Kato for valuable comments and encouragement. This work is supported by Grant-in-Aid for JSPS Fellows Grant Number 12J02113.

\section{Hilbert scheme of points in the affine plane}

\subsection{Results of Lehn-Sorger and Vasserot}

In \cite{LS} and \cite{V}, Lehn-Sorger and Vasserot give a description of the cohomology ring of $\mr{Hilb}^{n}(\bb{C}^{2})$ as the center of the group ring of the symmetric group $\mca{Z}(\bb{C}[\mf{S}_{n}])$. In this section, we recall some results of \cite{LS} and \cite{V}. 

For a partition $\hlam=(1^{\hal_{1}}2^{\hal_{2}}\ldots)$, we denote by $\ell(\hlam):=\sum_{i}\hal_{i}$ the length of $\hlam$ and $|\hlam|:=\sum_{i}i\hal_{i}$. For a partition $\hlam$ with $|\hlam|=n$, we denote by $\mf{C}_{\hlam}\subset\mf{S}_{n}$ the conjugacy class of $\mf{S}_{n}$ consisting of permutations whose cycle types are $\hlam$. The number of elements of $\mf{C}_{\hlam}$ is given by $$\#\mf{C}_{\hlam}=\frac{n!}{\prod_{i}i^{\hal_{i}}\hal_{i}!}.$$ We define the characteristic function $\chi_{\hlam}\in\mca{Z}(\bb{C}[\mf{S}_{n}])$ of $\mf{C}_{\hlam}$ by $$\chi_{\hlam}=\sum_{\sigma\in\mf{C}_{\hlam}}\sigma.$$ Then $\{\chi_{\hlam}\}_{\hlam\vdash n}$ forms an basis of $\mca{Z}(\bb{C}[\mf{S}_{n}])$. For $\sigma\in\mf{S}_{n}$, we define its degree by $\deg(\sigma)=2(n-\ell(\hlam))$ if $\sigma\in\mf{C}_{\hlam}$. Let $\bb{C}[\mf{S}_{n}]_{d}$ be the subspace of $\bb{C}[\mf{S}_{n}]$ spanned by $\sigma$ with $\deg(\sigma)=d$ and $$F^{d}\bb{C}[\mf{S}_{n}]:=\bigoplus_{d'\leq d}\bb{C}[\mf{S}_{n}]_{d'}.$$ Then $F^{d}\bb{C}[\mf{S}_{n}]$ defines a filtration on $\bb{C}[\mf{S}_{n}]$ compatible with the product. The induced product on $\gr^{F}\bb{C}[\mf{S}_{n}]$ is called the cup product and denoted by $\cup$. Since $\chi_{\hlam}$ is homogeneous, $\mca{Z}(\bb{C}[\mf{S}_{n}])$ inherits from $\bb{C}[\mf{S}_{n}]$ the gradation, filtration, and the cup product.

\begin{thm}[\cite{LS},\cite{V}]
There is an isomorphism of graded algebras $$H^{\ast}(\mr{Hilb}^{n}(\bb{C}^{2}),\bb{C})\cong\gr^{F}\mca{Z}(\bb{C}[\mf{S}_{n}]).$$ 
\end{thm}

We prove Theorem \ref{MainThm} by identifying $\gr^{F}\mca{Z}(\bb{C}[\mf{S}_{n}])$ and $\bb{C}[(S^{n}\bb{C}^{2})^{\bb{T}}]$. For later use, we give a formula for the cup product with $\chi_{(k+1,1^{n-k-1})}$ for $1\leq k\leq n-1$. For two partitions $\hmu=(1^{\hbeta_{1}}2^{\hbeta_{2}}\ldots)$ and $\hnu=(1^{\hgam_{1}}2^{\hgam_{2}}\ldots)$, we write $\hmu\preceq\hnu$ if $\hbeta_{i}\leq\hgam_{i}$ for any $i$. Note that this partial order is not related to the usual partial order on the set of partitions. 

\begin{lem}
For $\hlam=(1^{\hal_{1}}2^{\hal_{2}}\ldots)\vdash n$, we have $$\chi_{(k+1,1^{n-k-1})}\cup\chi_{\hlam}=\sum_{\substack{\hnu=(1^{\hgam_{1}}2^{\hgam_{2}}\ldots)\\ \ell(\hnu)=k+1, \hnu\preceq\hlam}}\frac{k!|\hnu|(\hal_{|\hnu|}+1)}{\prod_{i}\hgam_{i}!}\chi_{\hlam_{\hnu}}.$$ Here, $\hnu$ runs over partitions with $\ell(\hnu)=k+1$ and $\hnu\preceq\hlam$, and $\hlam_{\hnu}=(1^{\hbeta_{1}}2^{\hbeta_{2}}\ldots)\vdash n$ is defined by 
\begin{align*}
\hbeta_{i}=\begin{cases}\hal_{i}-\hgam_{i} & \mbox{ if }i\neq|\hnu| \\ \hal_{|\hnu|}+1 & \mbox{ if }i=|\hnu|.\end{cases}
\end{align*} 
\end{lem}

\begin{proof}
For $\sigma\in\mf{C}_{(k+1,1^{n-k-1})}$ (say $\sigma=(123\ldots k+1)$) and $\tau\in\mf{C}_{\hlam}$, the equality $\deg(\sigma\tau)=\deg(\sigma)+\deg(\tau)$ holds if and only if $1,2,\ldots,k+1$ are contained in different cycles for the disjoint cycle decomposition of $\tau$. If the number of elements of $\{1,2,\ldots,k+1\}$ which are contained in cycle of length $i$ is $\hgam_{i}$, then the cycle type of $\sigma\tau$ is given by $\hlam_{\hnu}$. For a fixed $\sigma\in\mf{C}_{(k+1,1^{n-k-1})}$, one can see that the number of $\tau\in\mf{C}_{\hlam}$ such that the cycle type of $\sigma\tau$ equals to $\hlam_{\hnu}$ is given by $$\frac{(k+1)!(n-k-1)!}{\prod_{i}i^{\hal_{i}-\hgam_{i}}\hgam_{i}!(\hal_{i}-\hgam_{i})!}.$$ Hence we have 

\begin{align*}
\chi_{(k+1,1^{n-k-1})}\cup\chi_{\hlam}&=\sum_{\substack{\hnu=(1^{\hgam_{1}}2^{\hgam_{2}}\ldots)\\ \ell(\hnu)=k+1, \hgam_{i}\leq\hal_{i}}}\frac{(k+1)!(n-k-1)!}{\prod_{i}i^{\hal_{i}-\hgam_{i}}\hgam_{i}!(\hal_{i}-\hgam_{i})!}\frac{\#\mf{C}_{(k+1,1^{n-k-1})}}{\#\mf{C}_{\hlam_{\hnu}}}\chi_{\hlam_{\hnu}} \\ 
&=\sum_{\hnu}\frac{(k+1)!(n-k-1)!}{\prod_{i}i^{\hal_{i}-\hgam_{i}}\hgam_{i}!(\hal_{i}-\hgam_{i})!}\frac{n!}{(k+1)(n-k-1)!}\frac{\prod_{i}i^{\hbeta_{i}}\hbeta_{i}!}{n!}\chi_{\hlam_{\hnu}} \\
&=\sum_{\hnu}\frac{k!|\hnu|(\hal_{|\hnu|}+1)}{\prod_{i}\hgam_{i}!}\chi_{\hlam_{\hnu}},
\end{align*}
which completes the proof.
\end{proof}

\subsection{MacMahon symmetric functions}

In this section, we study the ring structure of $\bb{C}[(S^{n}\bb{C}^{2})^{\bb{T}}]$. First we prepare some notation on symmetric functions in two set of variables (called MacMahon symmetric functions in $\cite{G}$). For an element $(a,b)\in\bb{N}\times\bb{N}$, an unordered sequence of vectors $\Lambda=(a_{1},b_{1})(a_{2},b_{2})\ldots(a_{l},b_{l})$ is called a bipartite partition of $(a,b)$ if $(a_{i},b_{i})\in\bb{N}\times\bb{N}\setminus\{(0,0)\}$ for any $i$ and $\sum_{i=1}^{l}a_{i}=a$, $\sum_{i=1}^{l}b_{i}=b$. We set $\ell(\Lambda)=l$ and $|\Lambda|=(a,b)$. We have a natural surjection $$\bb{C}[S^{n+1}\bb{C}^{2}]=\bb{C}[x_{1},\ldots,x_{n+1},y_{1},\ldots,y_{n+1}]^{\mf{S}_{n+1}}\twoheadrightarrow\bb{C}[S^{n}\bb{C}^{2}]=\bb{C}[x_{1},\ldots,x_{n},y_{1},\ldots,y_{n}]^{\mf{S}_{n}}$$ induced by $x_{i},y_{i}\mapsto x_{i},y_{i}$ for $1\leq i\leq n$ and $x_{n+1},y_{n+1}\mapsto 0$. We consider the projective limit $$S:=\varprojlim_{n}\bb{C}[S^{n}\bb{C}^2]$$ with respect to these surjections. This is the ring of MacMahon symmetric functions. 

For a bipartite partition $\Lambda=(a_{1},b_{1})\ldots(a_{l},b_{l})$, we define the monomial symmetric function $m_{\Lambda}\in S$ by symmetrization of the monomial $x_{1}^{a_{1}}y_{1}^{b_{1}}\cdots x_{l}^{a_{l}}y_{l}^{b_{l}}$ with coefficients $0$ or $1$. For example, we have $$m_{(1,1)(1,1)}=\sum_{i<j}x_{i}y_{i}x_{j}y_{j}.$$ It is clear that the monomial symmetric functions form an basis of $S$. Just as power sum symmetric functions freely generate the ring of symmetric functions, $m_{(a,b)}$'s generate $S$ as a $\bb{C}$-algebra and they are algebraically independent (\cite{D}). Hence we have $$S\cong\bb{C}[m_{(a,b)}\mid (a,b)\in\bb{N}\times\bb{N}\setminus\{(0,0)\}].$$ The kernel of the natural surjection $S\twoheadrightarrow\bb{C}[S^{n}\bb{C}^{2}]$ is generated by $\{m_{\Lambda}\mid\ell(\Lambda)>n\}$ as an ideal or a vector space. If $|\Lambda|=(a,b)$, then the weight of $m_{\Lambda}$ with respect to the $\bb{T}$-action induced by $t\cdot(x,y)=(t^{-1}x,ty)$ for $(x,y)\in\bb{C}^{2}$ is $a-b$. Hence the ideal of definition for the $\bb{T}$-fixed point scheme $(S^{n}\bb{C}^{2})^{\bb{T}}$ in $\bb{C}[S^{n}\bb{C}^{2}]$ is generated by the image of $\{m_{(a,b)}\mid a\neq b\}$ in $\bb{C}[S^{n}\bb{C}^{2}]$. Therefore, we have the following.

\begin{lem}
$$\bb{C}[(S^{n}\bb{C}^{2})^{\bb{T}}]\cong S/\left(m_{\Lambda}, m_{(a,b)} \mid\ell(\Lambda)>n, a\neq b \right).$$ In particular, $\bb{C}[(S^{n}\bb{C}^{2})^{\bb{T}}]$ is generated as a $\bb{C}$-algebra by $m_{(a,a)}$'s.
\end{lem}

For $(a,b)\in\bb{N}\times\bb{N}\setminus\{(0,0)\}$ and $\Lambda=(a_{1},b_{1})\ldots(a_{l},b_{l})$, we denote by $(a,b)\Lambda$ the bipartite partition $(a,b)(a_{1},b_{1})\ldots(a_{l},b_{l})$. If $(a,b)=(a_{i},b_{i})$ for some $i$, we denote by $\Lambda\setminus(a,b)$ the bipartite partition $(a_{1},b_{1})\ldots(a_{i-1},b_{i-1})(a_{i+1},b_{i+1})\ldots(a_{l},b_{l})$. We have $\ell(\Lambda\setminus(a,b))=\ell(\Lambda)-1$ and $|\Lambda\setminus(a,b)|=|\Lambda|-(a,b)$.

\begin{lem}
Let $\Lambda$ be a bipartite partition. For any $(i,j)\in\bb{N}\times\bb{N}\setminus\{(0,0)\}$, we denote by $c_{(i,j)}$ the multiplicity of $(i,j)$ in $\Lambda$. Then for $(a,b)\in\bb{N}\times\bb{N}\setminus\{(0,0)\}$, we have 
\begin{align*}
m_{(a,b)}m_{\Lambda}=(c_{(a,b)}+1)m_{(a,b)\Lambda}+\sum_{\substack{(i,j)\\ c_{(i,j)}>0}}(c_{(a+i,b+j)}+1)m_{(a+i,b+j)\Lambda\setminus(i,j)}.
\end{align*}
\end{lem}

\begin{proof}
For $(i,j)\in\bb{N}\times\bb{N}\setminus\{(0,0)\}$ with $c_{(i,j)}>0$, we set $c=c_{(a+i,b+j)}+1$ and consider the monomial $$x_{1}^{a+i}y_{1}^{b+j}\cdots x_{c}^{a+i}y_{c}^{b+j}\times(\mbox{monomial in }x_{i}\mbox{ and }y_{i}\mbox{ for }i>c)$$ in $m_{(a+i,b+j)\Lambda\setminus(i,j)}$ and its coefficient in the expansion of $m_{(a,b)}m_{\Lambda}$. In the expansion, this monomial appears as 
\begin{align*}
&x_{1}^{a}y_{1}^{b}\cdot x_{1}^{i}y_{1}^{j}x_{2}^{a+i}y_{2}^{b+j}\cdots x_{c}^{a+i}y_{c}^{b+j}\cdots, \\
&x_{2}^{a}y_{2}^{b}\cdot x_{1}^{a+i}y_{1}^{b+j}x_{2}^{i}y_{2}^{j}x_{3}^{a+i}y_{3}^{b+j}\cdots x_{c}^{a+i}y_{c}^{b+j}\cdots, \\
&\hspace{5em}\cdots\cdots \\
&x_{c}^{a}y_{c}^{b}\cdot x_{1}^{a+i}y_{1}^{b+j}x_{2}^{a+i}y_{2}^{b+j}\cdots x_{c-1}^{a+i}y_{c-1}^{b+j}x_{c}^{i}y_{c}^{j}\cdots.
\end{align*}
Hence the coefficient of $m_{(a+i,b+j)\Lambda\setminus(i,j)}$ in $m_{(a,b)}m_{\Lambda}$ is given by $c=c_{(a+i,b+j)}+1$. The coefficient of $m_{(a,b)\Lambda}$ can be understood in the same way.
\end{proof}

We denote by $\bm_{\Lambda}$ the image of $m_{\Lambda}$ under $S\twoheadrightarrow\bar{S}:=S/(m_{(a,b)}\mid a\neq b)$. For a partition $\lambda=(\lambda_{1},\lambda_{2},\ldots,\lambda_{l})$, we denote by $(\lambda,0)$ the bipartite partition $(\lambda_{1},0)(\lambda_{2},0)\ldots(\lambda_{l},0).$

\begin{lem}
$\{\bm_{(\lambda,0)(0,1)^{|\lambda|}}\mid\lambda\mbox{ : partition}\}$ forms a basis of $\bar{S}$.
\end{lem}

\begin{proof}
We first prove that the monomial symmetric functions of the form $\bm_{(a,a)(b,b)(c,c)\ldots}$ span $\bar{S}$. If $\Lambda$ does not contain $(a,b)$ with $a\neq b$, then $\bm_{\Lambda}$ is already of the form $\bm_{(a,a)(b,b)(c,c)\ldots}$. If $\Lambda$ contains $(a,b)$ with $a\neq b$, then we can expand $\bm_{\Lambda}$ in terms of $\bm_{\Phi}$ with $\ell(\Phi)=\ell(\Lambda)-1$ by Lemma 2.4 and $\bm_{(a,b)}=0$. By induction on the length of $\Lambda$, we get an expansion of $\bm_{\Lambda}$ in terms of monomial symmetric functions of the form $\bm_{(a,a)(b,b)(c,c)\ldots}$.

We next prove that we can expand $\bm_{(a,a)(b,b)(c,c)\ldots}$ in terms of $\bm_{(\lambda,0)(0,1)^{|\lambda|}}$'s. More generally, we show that monomial symmetric functions of the form $\bm_{(a_{1},b_{1})\ldots(a_{l},b_{l})(0,1)^{m}}$ with $a_{i}\geq b_{i}$ for any $i$ can be written as a linear combination of $\bm_{(\lambda,0)(0,1)^{|\lambda|}}$'s. We prove this claim by induction on $d=\sum_{i}b_{i}$. If $d=0$, then there is nothing to prove. 

Assume $d>0$ and the claim holds for smaller $d$. Then at least one of $b_{i}$ is positive. We can assume $b_{1}>0$. By Lemma 2.4, we have 
\begin{align*}
m_{(a_{1},b_{1}-1)}m_{(a_{2},b_{2})\ldots(a_{l},b_{l})(0,1)^{m+1}}&=c_{0}m_{(a_{1},b_{1}-1)(a_{2},b_{2})\ldots(a_{l},b_{l})(0,1)^{m+1}}+c_{1}m_{(a_{1},b_{1})\ldots(a_{l},b_{l})(0,1)^{m}} \\ &\hspace{2em}+\sum_{i=2}^{l}c_{i}m_{(a_{2},b_{2})\ldots(a_{1}+a_{i},b_{1}+b_{i}-1)\ldots(a_{l},b_{l})(0,1)^{m+1}}
\end{align*}
for some $c_{0},c_{1},\ldots,c_{l}\in\bb{Z}$ with $c_{1}\neq0$. By $\bm_{(a_{1},b_{1}-1)}=0$ and the induction hypothesis, we can expand $\bm_{(a_{1},b_{1})\ldots(a_{l},b_{l})(0,1)^{m}}$ in terms of $\bm_{(\lambda,0)(0,1)^{|\lambda|}}$'s. 

Since we have $\bar{S}\cong\bb{C}[m_{(a,a)}\mid a\in\bb{Z}_{>0}]$, the dimension of the degree $2k$-component of $\bar{S}$ is given by the number of partitions of $k$. Hence $\{\bm_{(\lambda,0)(0,1)^{|\lambda|}}\mid|\lambda|=k\}$ is linearly independent. This proves the lemma.
\end{proof}

In order to simplify some formulas, we understand that $1/x!=0$ for $x<0$ in the below.
\begin{lem}
For $k\geq l>0$ and $\lambda=(1^{\alpha_{1}}2^{\alpha_{2}}\ldots)$, we have 
\begin{align*}
\bm_{(k,l)(\lambda,0)(0,1)^{|\lambda|+k-l}}=\sum_{\mu=(1^{\beta_{1}}2^{\beta_{2}}\ldots)}\frac{(-1)^{l}l!(\beta_{|\lambda|-|\mu|+k}+1)}{(l-\ell(\lambda)+\ell(\mu))!\prod_{i}(\alpha_{i}-\beta_{i})!}\bm_{(|\lambda|-|\mu|+k,0)(\mu,0)(0,1)^{|\lambda|+k}}.
\end{align*}
Here, $\mu$ runs over all the partitions. By the above convention, only partitions satisfying $\ell(\lambda)-\ell(\mu)\leq l$ and $\mu\preceq\lambda$ contribute. 
\end{lem}

\begin{proof}
We prove this formula by induction on $l$. Assume $l=1$. By using Lemma 2.4 for $(a,b)=(k,0)$ and $\bm_{(k,0)}=0$, we have $$\bm_{(k,1)(\lambda,0)(0,1)^{|\lambda|+k-1}}=-(\alpha_{k}+1)\bm_{(k,0)(\lambda,0)(0,1)^{|\lambda|+k}}-\sum_{i:\alpha_{i}>0}(\alpha_{k+i}+1)\bm_{(k+i,0)(\lambda\setminus i,0)(0,1)^{|\lambda|+k}}.$$ The first term is equal to the contribution of $\mu=\lambda$ and the second term is equal to the contribution of $\mu=\lambda\setminus i$. 

Assume $l\geq2$ and the formula holds for smaller $l$. By using Lemma 2.4 for $(a,b)=(k,l-1)$ and the induction hypothesis, we have 
\begin{align*}
\bm_{(k,l)(\lambda,0)(0,1)^{|\lambda|+k-l}}&=-\bm_{(k,l-1)(\lambda,0)(0,1)^{|\lambda|+k-l+1}}-\sum_{j:\alpha_{j}>0}\bm_{(k+j,l-1)(\lambda\setminus j,0)(0,1)^{|\lambda|+k-l+1}} \\ 
&=\sum_{\mu=(1^{\beta_{1}}2^{\beta_{2}}\ldots)}\bm_{(|\lambda|-|\mu|+k,0)(\mu,0)(0,1)^{|\lambda|+k}}\Biggl\{\frac{(-1)^{l}(l-1)!(\beta_{|\lambda|-|\mu|+k}+1)}{(l-\ell(\lambda)+\ell(\mu)-1)!\prod_{i}(\alpha_{i}-\beta_{i})!} \\
&\hspace{3em} +\sum_{j}\frac{(-1)^{l}(l-1)!(\beta_{|\lambda|-|\mu|+k}+1)}{(l-\ell(\lambda)+\ell(\mu))!(\alpha_{j}-\beta_{j}-1)!\prod_{i\neq j}(\alpha_{i}-\beta_{i})!}\Biggr\} \\ 
&=\sum_{\mu}\frac{(-1)^{l}l!(\beta_{|\lambda|-|\mu|+k}+1)}{(l-\ell(\lambda)+\ell(\mu))!\prod_{i}(\alpha_{i}-\beta_{i})!}\bm_{(|\lambda|-|\mu|+k,0)(\mu,0)(0,1)^{|\lambda|+k}}.
\end{align*}
Here, in the last equality, we used the formula 
\begin{align}\label{eq}
\frac{(n_{0}+\cdots+n_{a})!}{n_{0}!\cdots n_{a}!}=\sum_{j=0}^{a}\frac{(n_{0}+\cdots n_{a}-1)!}{(n_{j}-1)!\prod_{i\neq j}n_{i}!}
\end{align}
for $n_{0},\ldots,n_{a}\in\bb{Z}_{\geq0}$ with $n_{0}+\cdots+n_{a}>0$. We applied it for $n_{0}=l-\ell(\lambda)+\ell(\mu)$ and $n_{i}=\alpha_{i}-\beta_{i}$ for $i\geq1$.
\end{proof}

For $\mu=(1^{\beta_{1}}2^{\beta_{2}}\ldots)$, $\nu=(1^{\gamma_{1}}2^{\gamma_{2}}\ldots)$, and $x\in\bb{Z}_{\geq|\nu|}$, we set $$f^{\mu}_{\nu}(x)=\frac{(x-|\nu|)!(x-|\mu|+1)}{(x-|\nu|-\ell(\nu)+\ell(\mu)+1)!\prod_{i}(\gamma_{i}-\beta_{i})!}.$$ We remark that we have $f^{\mu}_{\nu}(x)=0$ if $\mu\npreceq\nu$ and $f^{\nu}_{\nu}(x)=1$. We denote the partition $(1^{\beta_{1}}\ldots(j-1)^{\beta_{j-1}}j^{\beta_{j}+1}(j+1)^{\beta_{j+1}}\ldots)$ by $\mu\cup j$. 

\begin{lem}
For $\mu=(1^{\beta_{1}}2^{\beta_{2}}\ldots)$, $\nu=(1^{\gamma_{1}}2^{\gamma_{2}}\ldots)$, and $x\in\bb{Z}_{\geq|\nu|}$, we have
$$f^{\mu}_{\nu}(x+1)-f^{\mu}_{\nu}(x)=\sum_{j}f^{\mu\cup j}_{\nu}(x).$$
\end{lem}

\begin{proof}
We calculate as:
\begin{align*}
f^{\mu}_{\nu}(x+1)-f^{\mu}_{\nu}(x)&=\frac{(x-|\nu|)!}{(x-|\nu|-\ell(\nu)+\ell(\mu)+2)!\prod_{i}(\gamma_{i}-\beta_{i})!} \\
&\hspace{1em}\times\bigl\{(x+1-|\nu|)(x-|\mu|+2)-(x-|\nu|-\ell(\nu)+\ell(\mu)+2)(x-|\mu|+1)\bigr\} \\
&=\frac{(x-|\nu|)!\bigl\{(\ell(\nu)-\ell(\mu))(x-|\mu|+1)-|\nu|+|\mu|\bigr\}}{(x-|\nu|-\ell(\nu)+\ell(\mu)+2)!\prod_{i}(\gamma_{i}-\beta_{i})!} \\
&=\frac{(x-|\nu|)!\bigl\{\sum_{j}(\gamma_{j}-\beta_{j})(x-|\mu|-j+1)\bigr\}}{(x-|\nu|-\ell(\nu)+\ell(\mu)+2)!\prod_{i}(\gamma_{i}-\beta_{i})!} \\
&=\sum_{j}\frac{(x-|\nu|)!(x-|\mu\cup j|+1)}{(x-|\nu|-\ell(\nu)+\ell(\mu\cup j)+1)!(\gamma_{j}-\beta_{j}-1)!\prod_{i\neq j}(\gamma_{i}-\beta_{i})!} \\
&=\sum_{j}f^{\mu\cup j}_{\nu}(x).
\end{align*}
\end{proof}

\begin{lem}
For $\mu=(1^{\beta_{1}}2^{\beta_{2}}\ldots)\preceq\lambda=(1^{\alpha_{1}}2^{\alpha_{2}}\ldots)$, we have
\begin{align*}
\sum_{\substack{\nu=(1^{\gamma_{1}}2^{\gamma_{2}}\ldots) \\ \mu\preceq\nu\preceq\lambda}}(-1)^{\ell(\nu)+\ell(\lambda)}f^{\mu}_{\nu}(k+|\lambda|)\frac{(\ell(\lambda)-\ell(\nu)+|\lambda|-|\nu|)!}{(|\lambda|-|\nu|)!\prod_{i}(\alpha_{i}-\gamma_{i})!}=\frac{k!}{(k-\ell(\lambda)+\ell(\mu))!\prod_{i}(\alpha_{i}-\beta_{i})!}
\end{align*}
\end{lem}

\begin{proof}
We prove this formula by induction on $\ell=\ell(\lambda)-\ell(\mu)$. If $\ell=0$, then the formula is trivial since $\mu=\nu=\lambda$ and $f^{\nu}_{\nu}(x)=1$. Let us assume $\ell>0$ and the formula holds for smaller $\ell$. We set $$F^{\mu}_{\lambda}(k)=\sum_{\substack{\nu=(1^{\gamma_{1}}2^{\gamma_{2}}\ldots) \\ \mu\preceq\nu\preceq\lambda}}(-1)^{\ell(\nu)+\ell(\lambda)}f^{\mu}_{\nu}(k+|\lambda|)\frac{(\ell(\lambda)-\ell(\nu)+|\lambda|-|\nu|)!}{(|\lambda|-|\nu|)!\prod_{i}(\alpha_{i}-\gamma_{i})!}.$$ By Lemma 2.7 and the induction hypothesis, we have 
\begin{align*}
F^{\mu}_{\lambda}(k+1)-F^{\mu}_{\lambda}(k)&=\sum_{j}F^{\mu\cup j}_{\lambda}(k) \\
&=\sum_{j}\frac{k!}{(k-\ell(\lambda)+\ell(\mu)+1)!(\alpha_{j}-\beta_{j}-1)!\prod_{i\neq j}(\alpha_{i}-\beta_{i})!} \\
&=\frac{(k+1)!}{(k+1-\ell(\lambda)+\ell(\mu))!\prod_{i}(\alpha_{i}-\beta_{i})!}-\frac{k!}{(k-\ell(\lambda)+\ell(\mu))!\prod_{i}(\alpha_{i}-\beta_{i})!}.
\end{align*}
Here, the last equality follows from (\ref{eq}). Hence it is enough to prove the case of $k=0$, that is, $F^{\mu}_{\lambda}(0)=0$ since we assumed $\ell(\lambda)-\ell(\mu)>0$. We set $n_{i}=\alpha_{i}-\beta_{i}$ and $m_{i}=\alpha_{i}-\gamma_{i}$. Then 
\begin{align*}
F^{\mu}_{\lambda}(0)&=\sum_{\mu\preceq\nu\preceq\lambda}(-1)^{\ell(\nu)+\ell(\lambda)}\frac{(|\lambda|-|\nu|)!(|\lambda|-|\mu|+1)(\ell(\lambda)-\ell(\nu)+|\lambda|-|\nu|)!}{(|\lambda|-|\nu|-\ell(\nu)+\ell(\mu)+1)!(|\lambda|-|\nu|)!\prod_{i}(\gamma_{i}-\beta_{i})!(\alpha_{i}-\gamma_{i})!} \\
&=(|\lambda|-|\mu|+1)\sum_{0\leq m_{i}\leq n_{i}}(-1)^{\sum_{i}m_{i}}\frac{(\sum_{i}(i+1)m_{i})!}{(\sum_{i}(i+1)m_{i}-\sum_{i}n_{i}+1)!\prod_{i}m_{i}!(n_{i}-m_{i})!} \\
&=\frac{(|\lambda|-|\mu|+1)(\sum_{i}n_{i}-1)!}{\prod_{i}n_{i}!}\sum_{0\leq m_{i}\leq n_{i}}(-1)^{\sum_{i}m_{i}}\binom{\sum_{i}(i+1)m_{i}}{\sum_{i}n_{i}-1}\prod_{i}\binom{n_{i}}{m_{i}} \\
\end{align*}

Let us consider the coefficient of $x^{\sum_{i}n_{i}-1}$ in the expansion of $\prod_{i}((x+1)^{i+1}-1)^{n_{i}}$. Since we have $\prod_{i}((x+1)^{i+1}-1)^{n_{i}}=\prod_{i}(i+1)^{n_{i}}\cdot x^{\sum_{i}n_{i}}+(\mbox{higher order terms})$, the coefficient of $x^{\sum_{i}n_{i}-1}$ is $0$. On the other hand, we calculate as: 
\begin{align*}
\prod_{i}((x+1)^{i+1}-1)^{n_{i}}&=\prod_{i}\left(\sum_{0\leq m_{i}\leq n_{i}}(-1)^{n_{i}-m_{i}}\binom{n_{i}}{m_{i}}(x+1)^{(i+1)m_{i}}\right)\\
&=\sum_{0\leq m_{i}\leq n_{i}}(-1)^{\sum_{i}(n_{i}-m_{i})}\prod_{i}\binom{n_{i}}{m_{i}}\cdot(x+1)^{\sum_{i}(i+1)m_{i}}\\
&=\sum_{k\geq0}\sum_{0\leq m_{i}\leq n_{i}}(-1)^{\sum_{i}(n_{i}-m_{i})}\binom{\sum_{i}(i+1)m_{i}}{k}\prod_{i}\binom{n_{i}}{m_{i}}\cdot x^{k}.
\end{align*}
Hence we have $$\sum_{0\leq m_{i}\leq n_{i}}(-1)^{\sum_{i}m_{i}}\binom{\sum_{i}(i+1)m_{i}}{\sum_{i}n_{i}-1}\prod_{i}\binom{n_{i}}{m_{i}}=0.$$ This implies $F^{\mu}_{\lambda}(0)=0.$
\end{proof}

\begin{lem}
For any $k\geq0$, $l>0$, and $\lambda=(1^{\alpha_{1}}2^{\alpha_{2}}\ldots)$ with $|\lambda|+k-l\geq0$, we have 
\begin{align*}
\bm_{(k,l)(\lambda,0)(0,1)^{|\lambda|+k-l}}&=\sum_{\substack{\mu=(1^{\beta_{1}}2^{\beta_{2}}\ldots)\\ \mu\preceq\lambda}}(\beta_{|\lambda|-|\mu|+k}+1)\bm_{(|\lambda|-|\mu|+k,0)(\mu,0)(0,1)^{|\lambda|+k}} \\
&\hspace{0em}\times\Biggl\{\sum_{\substack{\nu=(1^{\gamma_{1}}2^{\gamma_{2}}\ldots) \\ \mu\preceq\nu\preceq\lambda \\ |\nu|\leq|\lambda|+k-l}}(-1)^{\ell(\nu)+\ell(\lambda)+l}f^{\mu}_{\nu}(k+|\lambda|)\frac{(\ell(\lambda)-\ell(\nu)+|\lambda|-|\nu|+k-l)!}{(|\lambda|-|\nu|+k-l)!\prod_{i}(\alpha_{i}-\gamma_{i})!}\Biggr\} 
\end{align*}
\end{lem}

\begin{proof}
We prove this formula by induction on $|\lambda|+\ell(\lambda)+k-l\geq0$. Assume $|\lambda|+\ell(\lambda)+k-l=0$. It implies that $\lambda=\emptyset$ and $k=l$. Then the formula reduces to 
\begin{align}\label{kk}
\bm_{(k,k)}=(-1)^{k}\bm_{(k,0)(0,1)^{k}}.
\end{align}
This is a special case of Lemma 2.6. 

Let us assume $|\lambda|+\ell(\lambda)+k-l>0$. If $k=l$, then by Lemma 2.6 and Lemma 2.8, we have
\begin{align*}
\bm_{(k,k)(\lambda,0)(0,1)^{|\lambda|}}&=\sum_{\mu=(1^{\beta_{1}}2^{\beta_{2}}\ldots)}\frac{(-1)^{k}k!(\beta_{|\lambda|-|\mu|+k}+1)}{(k-\ell(\lambda)+\ell(\mu))!\prod_{i}(\alpha_{i}-\beta_{i})!}\bm_{(|\lambda|-|\mu|+k,0)(\mu,0)(0,1)^{|\lambda|+k}}\\
&=\sum_{\substack{\mu=(1^{\beta_{1}}2^{\beta_{2}}\ldots)\\ \mu\preceq\lambda}}(\beta_{|\lambda|-|\mu|+k}+1)\bm_{(|\lambda|-|\mu|+k,0)(\mu,0)(0,1)^{|\lambda|+k}} \\
&\hspace{0em}\times\Biggl\{\sum_{\substack{\nu=(1^{\gamma_{1}}2^{\gamma_{2}}\ldots) \\ \mu\preceq\nu\preceq\lambda}}(-1)^{\ell(\nu)+\ell(\lambda)+k}f^{\mu}_{\nu}(k+|\lambda|)\frac{(\ell(\lambda)-\ell(\nu)+|\lambda|-|\nu|)!}{(|\lambda|-|\nu|)!\prod_{i}(\alpha_{i}-\gamma_{i})!}\Biggr\}.
\end{align*}
This implies the formula.

If $k\neq l$, then by Lemma 2.4 and the induction hypothesis, we have 
\begin{align*}
\bm_{(k,l)(\lambda,0)(0,1)^{|\lambda|+k-l}}&=-\bm_{(k,l+1)(\lambda,0)(0,1)^{|\lambda|+k-l-1}}-\sum_{j:\alpha_{j}>0}\bm_{(k+j,l)(\lambda\setminus j,0)(0,1)^{|\lambda|+k-l}} \\
&\hspace{-5em}=\sum_{\mu}(\beta_{|\lambda|-|\mu|+k}+1)\bm_{(|\lambda|-|\mu|+k,0)(\mu,0)(0,1)^{|\lambda|+k}} \\ 
&\hspace{-3em}\times\Biggl\{\sum_{\substack{\nu \\ \mu\preceq\nu\preceq\lambda \\ |\nu|\leq|\lambda|+k-l-1}}\frac{(-1)^{\ell(\nu)+\ell(\lambda)+l}f^{\mu}_{\nu}(k+|\lambda|)(\ell(\lambda)-\ell(\nu)+|\lambda|-|\nu|+k-l-1)!}{(|\lambda|-|\nu|+k-l-1)!\prod_{i}(\alpha_{i}-\gamma_{i})!} \\
&\hspace{-3em}+\sum_{j}\sum_{\substack{\nu \\ \mu\preceq\nu\preceq\lambda\setminus j \\ |\nu|\leq|\lambda|+k-l}}\frac{(-1)^{\ell(\nu)+\ell(\lambda)+l}f^{\mu}_{\nu}(k+|\lambda|)(\ell(\lambda)-\ell(\nu)+|\lambda|-|\nu|+k-l-1)!}{(|\lambda|-|\nu|+k-l)!(\alpha_{j}-\gamma_{j}-1)!\prod_{i\neq j}(\alpha_{i}-\gamma_{i})!}\Biggr\} \\
&\hspace{-5em}=\sum_{\mu}(\beta_{|\lambda|-|\mu|+k}+1)\bm_{(|\lambda|-|\mu|+k,0)(\mu,0)(0,1)^{|\lambda|+k}} \\
&\hspace{-3em}\times\Biggl\{\sum_{\substack{\nu \\ \mu\preceq\nu\preceq\lambda \\ |\nu|\leq|\lambda|+k-l}}(-1)^{\ell(\nu)+\ell(\lambda)+l}f^{\mu}_{\nu}(k+|\lambda|)\frac{(\ell(\lambda)-\ell(\nu)+|\lambda|-|\nu|+k-l)!}{(|\lambda|-|\nu|+k-l)!\prod_{i}(\alpha_{i}-\gamma_{i})!}\Biggr\}. 
\end{align*}
Here, the last equality follows from (\ref{eq}). This completes the proof of the formula.
\end{proof}

\begin{lem}
For $k\in\bb{Z}_{>0}$ and $\lambda=(1^{\alpha_{1}}2^{\alpha_{2}}\ldots)$, we have 
\begin{align*}
\bm_{(k,k)}\bm_{(\lambda,0)(0,1)^{|\lambda|}}=\sum_{\substack{\mu=(1^{\beta_{1}}2^{\beta_{2}}\ldots)\preceq\lambda\\ \ell(\lambda)-\ell(\mu)\leq k+1}}\frac{(-1)^{k}k!(k+|\lambda|-|\mu|+1)(\beta_{|\lambda|-|\mu|+k}+1)}{(k-\ell(\lambda)+\ell(\mu)+1)!\prod_{i}(\alpha_{i}-\beta_{i})!}\bm_{(|\lambda|-|\mu|+k,0)(\mu,0)(0,1)^{|\lambda|+k}}
\end{align*}
\end{lem}

\begin{proof}
By using Lemma 2.4 and Lemma 2.9, we calculate as:
\begin{align*}
\bm_{(k,k)}\bm_{(\lambda,0)(0,1)^{|\lambda|}}&=\bm_{(k,k)(\lambda,0)(0,1)^{|\lambda|}}+\bm_{(k,k+1)(\lambda,0)(0,1)^{|\lambda|-1}}+\sum_{j}\bm_{(k+j,k)(\lambda\setminus j,0)(0,1)^{|\lambda|}} \\
&\hspace{-5em}=\sum_{\mu=(1^{\beta_{1}}2^{\beta_{2}}\ldots)\preceq\lambda}(\beta_{|\lambda|-|\mu|+k}+1)\bm_{(|\lambda|-|\mu|+k,0)(\mu,0)(0,1)^{|\lambda|+k}} \\
&\hspace{-3em}\times\Biggl\{\sum_{\substack{\nu=(1^{\gamma_{1}}2^{\gamma_{2}}\ldots) \\ \mu\preceq\nu\preceq\lambda}}(-1)^{\ell(\nu)+\ell(\lambda)+k}f^{\mu}_{\nu}(k+|\lambda|)\frac{(\ell(\lambda)-\ell(\nu)+|\lambda|-|\nu|)!}{(|\lambda|-|\nu|)!\prod_{i}(\alpha_{i}-\gamma_{i})!} \\
&\hspace{-3em}-\sum_{\substack{\nu=(1^{\gamma_{1}}2^{\gamma_{2}}\ldots) \\ \mu\preceq\nu\preceq\lambda \\ |\nu|\leq|\lambda|-1}}(-1)^{\ell(\nu)+\ell(\lambda)+k}f^{\mu}_{\nu}(k+|\lambda|)\frac{(\ell(\lambda)-\ell(\nu)+|\lambda|-|\nu|-1)!}{(|\lambda|-|\nu|-1)!\prod_{i}(\alpha_{i}-\gamma_{i})!} \\
&\hspace{-3em}-\sum_{j}\sum_{\substack{\nu=(1^{\gamma_{1}}2^{\gamma_{2}}\ldots) \\ \mu\preceq\nu\preceq\lambda\setminus j}}(-1)^{\ell(\nu)+\ell(\lambda)+k}f^{\mu}_{\nu}(k+|\lambda|)\frac{(\ell(\lambda)-\ell(\nu)+|\lambda|-|\nu|-1)!}{(|\lambda|-|\nu|)!(\alpha_{j}-\gamma_{j}-1)!\prod_{i\neq j}(\alpha_{i}-\gamma_{i})!}\Biggr\} \\
&\hspace{-5em}=\sum_{\mu=(1^{\beta_{1}}2^{\beta_{2}}\ldots)\preceq\lambda}\frac{(-1)^{k}f^{\mu}_{\lambda}(k+|\lambda|)(\beta_{|\lambda|-|\mu|+k}+1)}{\prod_{i}(\alpha_{i}-\beta_{i})!}\bm_{(|\lambda|-|\mu|+k,0)(\mu,0)(0,1)^{|\lambda|+k}}
\end{align*}
Here, the last equality follows from (\ref{eq}).
\end{proof}

\begin{cor}
For $k\in\bb{Z}_{>0}$ and a partition $\lambda$, $\bm_{(k,k)}\bm_{(\lambda,0)(0,1)^{|\lambda|}}$ is contained in the linear span of $\{\bm_{(\nu,0)(0,1)^{|\nu|}}\mid\ell(\nu)+|\nu|\geq\ell(\lambda)+|\lambda|\}$.
\end{cor}

\begin{proof}
If $\bm_{(|\lambda|-|\mu|+k,0)(\mu,0)(0,1)^{|\lambda|+k}}$ appears in the formula for $\bm_{(k,k)}\bm_{(\lambda,0)(0,1)^{|\lambda|}}$ in Lemma 2.10, then we have $\ell(\lambda)-\ell(\mu)\leq k+1$ and hence
\begin{align*}
\ell(\mu\cup(|\lambda|-|\mu|+k))+|\mu\cup(|\lambda|-|\mu|+k)|&=\ell(\mu)+1+|\mu|+|\lambda|-|\mu|+k \\
&\geq \ell(\lambda)+|\lambda|
\end{align*}
This implies the corollary.
\end{proof}

\subsection{Proof of Theorem \ref{MainThm}}

\begin{lem}
The image of $\{m_{(\lambda,0)(0,1)^{|\lambda|}}\mid \ell(\lambda)+|\lambda|\leq n\}$ in $\bb{C}[(S^{n}\bb{C}^{2})^{\bb{T}}]$ forms a basis of $\bb{C}[(S^{n}\bb{C}^{2})^{\bb{T}}]$.
\end{lem}

\begin{proof}
By Lemma 2.3, the kernel of the natural surjection $\bar{S}\twoheadrightarrow\bb{C}[(S^{n}\bb{C}^{2})^{\bb{T}}]$ is spanned by $\{\bm_{\Lambda}\mid\ell(\Lambda)>n\}$. By Lemma 2.5, it suffices to show that each $\bm_{\Lambda}$ can be written as a linear combination of $\bm_{(\lambda,0)(0,1)^{|\lambda|}}$ with $\ell(\lambda)+|\lambda|\geq\ell(\Lambda)$. Set $\deg(\bm_{\Lambda})=2d(\Lambda)$ and set $e(\Lambda)$ to be the number of $(0,1)$ in $\Lambda$. We prove this claim by induction on $d(\Lambda)-e(\Lambda)$. If $e(\Lambda)=d(\Lambda)$, then there is nothing to prove. If $d(\Lambda)>e(\Lambda)$, then $\Lambda$ contains $(a,b)$ with $b>0$ and $(a,b)\neq(0,1)$. Let us write $\Lambda=(a,b)\Lambda'$. By Lemma 2.4, we have $$\bm_{(a,b-1)}\bm_{(0,1)\Lambda'}=c_{0}\bm_{(a,b-1)(0,1)\Lambda'}+c_{1}\bm_{\Lambda}+\sum_{\substack{\Lambda'' \\ e(\Lambda'')=e(\Lambda) \\ \ell(\Lambda'')=\ell(\Lambda)-1}}c_{\Lambda''}\bm_{(0,1)\Lambda''}$$ for some coefficients $c_{\ast}\in\bb{Z}$ with $c_{1}\neq0$. By the induction hypothesis, we have $\bm_{(a,b-1)(0,1)\Lambda'}$, $\bm_{(0,1)\Lambda''}\in\Bigl\langle\bm_{(\lambda,0)(0,1)^{|\lambda|}}\mid\ell(\lambda)+|\lambda|\geq\ell(\Lambda)\Bigr\rangle.$ If $b\neq a+1$, then this implies the claim since we have $\bm_{(a,b-1)}=0$. Let us assume $b=a+1$. Since we have $d((0,1)\Lambda')-e((0,1)\Lambda')=d(\Lambda)-e(\Lambda)-a-1$ and $\ell((0,1)\Lambda')=\ell(\Lambda)$, the induction hypothesis implies that $\bm_{(0,1)\Lambda'}\in\Bigl\langle\bm_{(\lambda,0)(0,1)^{|\lambda|}}\mid\ell(\lambda)+|\lambda|\geq\ell(\Lambda)\Bigr\rangle$. Then Corollary 2.11 implies that $\bm_{(a,a)}\bm_{(0,1)\Lambda'}\in\Bigl\langle\bm_{(\lambda,0)(0,1)^{|\lambda|}}\mid\ell(\lambda)+|\lambda|\geq\ell(\Lambda)\Bigr\rangle$. Therefore, we have $\bm_{\Lambda}\in\Bigl\langle\bm_{(\lambda,0)(0,1)^{|\lambda|}}\mid\ell(\lambda)+|\lambda|\geq\ell(\Lambda)\Bigr\rangle$ as required.

\end{proof}

\begin{proof}[Proof of Theorem \ref{MainThm}]
For a partition $\lambda=(1^{\alpha_{1}}2^{\alpha_{2}}\ldots)$ with $\ell(\lambda)+|\lambda|\leq n$, we denote by $\hlam=(1^{\hal_{1}}2^{\hal_{2}}\ldots)\vdash n$ the partition given by $\hal_{1}=n-\ell(\lambda)-|\lambda|$ and $\hal_{i}=\alpha_{i-1}$ for $i\geq2$. Let $\psi:\bb{C}[(S^{n}\bb{C}^{2})^{\bb{T}}]\rightarrow\gr^{F}\mca{Z}(\bb{C}[\mf{S}_{n}])$ be the linear map defined by $$\psi(\bm_{(\lambda,0)(0,1)^{|\lambda|}})=(-1)^{|\lambda|}\chi_{\hlam}.$$ By Lemma 2.12, $\psi$ is well-defined and an isomorphism of graded vector spaces. Since we have $\deg(\sigma)\leq 2(n-1)$ for any $\sigma\in\mf{S}_{n}$, $\{\bm_{(k,k)}\mid 1\leq k\leq n-1\}$ generates $\bb{C}[(S^{n}\bb{C}^{2})^{\bb{T}}]$ as a $\bb{C}$-algebra by Lemma 2.3. Hence it suffices to prove $$\psi(\bm_{(k,k)}\bm_{(\lambda,0)(0,1)^{|\lambda|}})=\psi(\bm_{(k,k)})\cup\psi(\bm_{(\lambda,0)(0,1)^{|\lambda|}})$$ for $1\leq k\leq n-1$. By Lemma 2.10, $\psi(\bm_{(k,k)}\bm_{(\lambda,0)(0,1)^{|\lambda|}})$ is given by $$\sum_{\substack{\mu=(1^{\beta_{1}}2^{\beta_{2}}\ldots) \preceq\lambda \\ \ell(\lambda)-\ell(\mu)\leq k+1}}\Biggl\{\frac{(-1)^{k}k!(k+|\lambda|-|\mu|+1)(\beta_{|\lambda|-|\mu|+k}+1)}{(k-\ell(\lambda)+\ell(\mu)+1)!\prod_{i}(\alpha_{i}-\beta_{i})!}\psi(\bm_{(|\lambda|-|\mu|+k,0)(\mu,0)(0,1)^{|\lambda|+k}})\Biggr\}.$$ Since $$\ell(\mu\cup(|\lambda|-|\mu|+k))+|\mu\cup(|\lambda|-|\mu|+k)|=|\lambda|+\ell(\mu)+k+1,$$ we have $\bm_{(|\lambda|-|\mu|+k,0)(\mu,0)(0,1)^{|\lambda|+k}}=0$ if $|\lambda|+\ell(\mu)+k+1>n$. Hence in the above sum, only $\mu$'s satisfying $|\lambda|+\ell(\mu)+k+1\leq n$ contribute.

For a partition $\mu=(1^{\beta_{1}}2^{\beta_{2}}\ldots)$ with $\ell(\lambda)-\ell(\mu)\leq k+1$, $\mu\preceq\lambda$, and $|\lambda|+\ell(\mu)+k+1\leq n$, we associate a partition $\xi(\mu)=(1^{\hgam_{1}}2^{\hgam_{2}}\ldots)$ with $\ell(\xi(\mu))=k+1$ by 
\begin{align*}
\hgam_{i}=\begin{cases}k+1-\ell(\lambda)+\ell(\mu)&\mbox{ if }i=1 \\ \alpha_{i-1}-\beta_{i-1}&\mbox{ if }i\geq 2.\end{cases}
\end{align*}
We have $\hgam_{1}\leq\hal_{1}$ and hence $\xi(\mu)\preceq\hlam$. This $\xi$ gives a bijection between the set of partitions $\mu$ with $\ell(\lambda)-\ell(\mu)\leq k+1$, $\mu\preceq\lambda$, and $|\lambda|+\ell(\mu)+k+1\leq n$ and the set of partitions $\hnu$ with $\ell(\hnu)=k+1$ and $\hnu\preceq\hlam$. We have $|\xi(\mu)|=k+1+|\lambda|-|\mu|$ and $\psi(\bm_{(|\lambda|-|\mu|+k,0)(\mu,0)(0,1)^{|\lambda|+k}})=(-1)^{|\lambda|+k}\chi_{\hlam_{\xi(\mu)}}$, where $\hlam_{\xi(\mu)}\vdash n$ is defined as in Lemma 2.2. By $\ell(\xi(\mu))=k+1>1$, we have $\hgam_{|\xi(\mu)|}=0$ and hence $\alpha_{|\xi(\mu)|-1}=\beta_{|\xi(\mu)|-1}$. Therefore, by Lemma 2.2, we have 
\begin{align*}
\psi(\bm_{(k,k)}\bm_{(\lambda,0)(0,1)^{|\lambda|}})&=\sum_{\substack{\mu=(1^{\beta_{1}}2^{\beta_{2}}\ldots)\preceq\lambda \\ \ell(\lambda)-\ell(\mu)\leq k+1 \\|\lambda|+\ell(\mu)+k+1\leq n}}\frac{(-1)^{|\lambda|}k!(k+|\lambda|-|\mu|+1)(\beta_{|\lambda|-|\mu|+k}+1)}{(k-\ell(\lambda)+\ell(\mu)+1)!\prod_{i}(\alpha_{i}-\beta_{i})!}\chi_{\hlam_{\hnu(\mu)}} \\
&=\sum_{\substack{\hnu=(1^{\hgam_{1}}2^{\hgam_{2}}\ldots)\preceq\hlam \\ \ell(\hnu)=k+1}}\frac{(-1)^{|\lambda|}k!|\hnu|(\hal_{|\hnu|}+1)}{\prod_{i}\hgam_{i}!}\chi_{\hlam_{\hnu}} \\
&=(-1)^{|\lambda|}\chi_{(k+1,1^{n-k-1})}\cup\chi_{\hlam} \\
&=\psi(\bm_{(k,k)})\cup\psi(\bm_{(\lambda,0)(0,1)^{|\lambda|}}).
\end{align*}
Here, in the last equality, we used (\ref{kk}). This completes the proof of Theorem \ref{MainThm}.

\end{proof}

\appendix

\section{Spaltenstein variety}

As in the introduction, let $G=\mr{GL}_{n}$ and $\mf{g}=\mf{gl}_{n}$. We fix a Cartan subalgebra $\mf{t}$ and a Borel subalgebra $\mf{b}\supset\mf{t}$. Let $\mathfrak{b}\subset\mf{p}, \mf{q}\subset\mf{g}$ be two standard parabolic subalgebras and $P$, $Q$ be the parabolic subgroups of $G$ with Lie algebras $\mf{p}$, $\mf{q}$. We denote a Levi and the nilpotent part of $\mf{p}$ (resp. $\mf{q}$) by $\mf{l}_{P}$ and $\mf{n}_{P}$ (resp. $\mf{l}_{Q}$ and $\mf{n}_{Q}$). Let $L_{Q}$ be the Levi subgroup of $G$ with its Lie algebra $\mf{l}_{Q}$. We take a regular nilpotent element $e_{P}$ of $\mf{l}_{P}$. We also take a regular nilpotent element $e_{Q}$ of $\mf{l}_{Q}$ and fix a $\mf{sl}_{2}$-triple $\{e_{Q}, h_{Q}, f_{Q}\}$. We denote the centralizer of $f_{Q}$ in $\mf{l}_{Q}$ (resp. in $\mf{g}$) by $Z_{\mf{l}_{Q}}(f_{Q})$ (resp. $Z_{\mf{g}}(f_{Q}))$. We set $\mca{N}_{P}=\mr{Ad}(G)\cdot\mf{n}_{P}$ and consider the scheme-theoretic intersection $\mca{N}_{P}\cap(e_{Q}+Z_{\mf{l}_{Q}}(f_{Q}))$ of $\mca{N}_{P}$ and $e_{Q}+Z_{\mf{l}_{Q}}(f_{Q})$ in $\mf{g}$. There is a $\bb{G}_{m}$-action on $\mca{N}_{P}\cap(e_{Q}+Z_{\mf{l}_{Q}}(f_{Q}))$ induced from the $\bb{G}_{m}$-action on $\mf{g}$ given by $t\cdot X=t^{-2}\mr{Ad}(t^{h_{Q}})X$ for $t\in\bb{G}_{m}$ and $X\in\mf{g}$. Let $\mca{X}^{Q}_{e_{P}}=\{gQ\in G/Q\mid\mr{Ad}(g)^{-1}e_{P}\in\mf{n}_{Q}\}$ be the Spaltenstein variety associated to $e_{P}$ and $Q$. 

\begin{thm}
There is a graded algebra isomorphism $$H^{\ast}(\mca{X}^{Q}_{e_{P}},\bb{C})\cong\bb{C}[\mca{N}_{P}\cap(e_{Q}+Z_{\mf{l}_{Q}}(f_{Q}))].$$ Here, the grading on $\bb{C}[\mca{N}_{P}\cap(e_{Q}+Z_{\mf{l}_{Q}}(f_{Q}))]$ comes from the $\bb{G}_{m}$-action above.
\end{thm}

If $Q=B$ and $e_{Q}=0$, then $\mca{X}^{Q}_{e_{P}}$ coincides with the Springer fiber $\mca{B}_{e_{P}}$ and the above description of its cohomology ring reduces to Theorem \ref{DPT}.

\begin{rem}
Let $Z(L_{Q})$ be the center of $L_{Q}$. Then $Z(L_{Q})$ acts on $\mca{N}_{P}\cap(e_{Q}+Z_{\mf{g}}(f_{Q}))$ by the adjoint action. One can easily see that the scheme-theoretic intersection $\mca{N}_{P}\cap(e_{Q}+Z_{\mf{l}_{Q}}(f_{Q}))$ is isomorphic to the $Z(L_{Q})$-fixed point scheme $(\mca{N}_{P}\cap(e_{Q}+Z_{\mf{g}}(f_{Q})))^{Z(L_{Q})}$ of $\mca{N}_{P}\cap(e_{Q}+Z_{\mf{g}}(f_{Q}))$. 
\end{rem}

For the proof of Theorem A.1, we use the presentation of the cohomology ring $H^{\ast}(\mca{X}^{Q}_{e_{P}},\bb{C})$ by Brundan-Ostrik \cite{BO} and the defining equations of $\mca{N}_{P}$ in $\mf{g}$ which was conjectured by Tanisaki \cite{T} and proved by Weyman \cite{W}. We first recall some results from \cite{BO}.

Let $\lambda=(\lambda_{1}\geq\ldots\geq\lambda_{n}\geq 0)$ be the transpose of the partition corresponding to $P$ and $\mu=(\mu_{1},\ldots,\mu_{n})$, $\mu_{i}\geq0$, the composition of $n$ corresponding to $Q$. Then $e_{P}$ is the nilpotent matrix whose Jordan block is of type $\lambda^{T}$, and $\mca{N}_{P}$ is the closure of the nilpotent orbit whose Jordan block is of type $\lambda$. 

Let $R:=\bb{C}[x_{1},\ldots,x_{n}]$ be the polynomial ring in $n$-variables. We define its grading by $\deg(x_{i})=2$. Let $\mf{S}_{\mu}:=\mf{S}_{\mu_{1}}\times\cdots\times\mf{S}_{\mu_{n}}$ be the parabolic subgroup of $n$-th symmetric group $\mf{S}_{n}$. For $1\leq i\leq l$ and $r\in\bb{Z}_{\geq 0}$, we denote by $e_{r}(\mu;i)$ the $r$-th elementary symmeric polynomial in the variables $\{x_{k}\mid \mu_{1}+\cdots+\mu_{i-1}+1\leq k\leq \mu_{1}+\cdots+\mu_{i}\}$. We also set $e_{0}(\mu;i)=1$. Then the algebra of $\mf{S}_{\mu}$-invariant polynomials $R_{\mu}:=R^{\mf{S}_{\mu}}$ is freely generated by $\{e_{r}(\mu;i)\mid 1\leq i\leq n, 1\leq r\leq\mu_{i}\}$. 

For $m\geq1$, $1\leq i_{1}<\cdots<i_{m}\leq n$ and $r\geq 0$, let $$e_{r}(\mu;i_{1},\ldots,i_{m}):=\sum_{r_{1}+\cdots+r_{m}=r}e_{r_{1}}(\mu;i_{1})\cdots e_{r_{m}}(\mu;i_{m}).$$ Let $I^{\lambda}_{\mu}$ be the ideal of $R_{\mu}$ generated by 
$$\left\{e_{r}(\mu;i_{1},\ldots,i_{m})\relmiddle|\begin{array}{l}m\geq 1,1\leq i_{1}<\cdots<i_{m}\leq n,\\ r>\mu_{i_{1}}+\cdots+\mu_{i_{m}}-\lambda_{a+1}-\cdots-\lambda_{n}\\ \textrm{where } a:=\#\{i\mid\mu_{i}>0,i\neq i_{1},\ldots,i_{m}\}\end{array}\right\}.$$ 

\begin{thm}[\cite{BO}]
There is an isomorphism of graded algebras $$H^{\ast}(\mca{X}^{Q}_{e_{P}},\bb{C})\cong R_{\mu}/I^{\lambda}_{\mu}.$$
\end{thm}

Next we recall the defining equations of nilpotent orbit closures of $\mf{g}$. We have $\bb{C}[\mf{g}]=\bb{C}[x_{ij}]$, where $x_{ij}$ $(1\leq i,j\leq n)$ is the $(i,j)$-th coordinate of matrices. Let $\{g^{\lambda}_{u}\}_{u}$ be the set of coefficients of $t^{k}$ in $s$-minors of $(tI-(x_{ij}))$ with $s=1,\ldots,n$ and $k<\lambda_{n-s+1}+\lambda_{n-s+2}+\cdots+\lambda_{n}$.  

\begin{thm}[\cite{W}]
The defining ideal of $\mca{N}_{P}$ in $\mf{g}$ is generated by $\{g^{\lambda}_{u}\}_{u}$.
\end{thm}

\begin{proof}[Proof of Theorem A.1]
First we prepare some notation. We define $\mu_{i}\times\mu_{i}$-matrices $E_{i}$ and $F_{i}$ as follows: 
$$E_{i}=\left(
\begin{array}{ccccc}
0&\cdots&\cdots&\cdots&0\\
1&\ddots&&&\vdots\\
0&\ddots&\ddots&&\vdots\\
\vdots&\ddots&\ddots&\ddots&\vdots\\
0&\cdots&0&1&0\\
\end{array}
\right),$$
$$F_{i}=\left(
\begin{array}{ccccc}
0&\hspace{-5pt}\mu_{i}-1&0&\cdots&0\\
0&\hspace{-5pt}0&\hspace{-10pt}\text{\small{$2(\mu_{i}-2)$}}&\ddots&\vdots\\
\vdots&\hspace{-5pt}\ddots&\hspace{-10pt}\ddots&\hspace{-10pt}\ddots&\vdots\\
\vdots&&\hspace{-10pt}\ddots&\hspace{-10pt}\ddots&0\\
\vdots&&&\hspace{-10pt}0&\mu_{i}-1\\
0&\hspace{-5pt}\cdots&\hspace{-10pt}\cdots&\hspace{-10pt}0&0\\
\end{array}
\right).
$$
We set $Z_{i}(x^{(i)})=Z_{i}(x^{(i)}_{1},\ldots,x^{(i)}_{\mu_{i}}):=E_{i}+x^{(i)}_{1}I+x^{(i)}_{2}F_{i}+x^{(i)}_{3}F^{2}_{i}+\cdots+x^{(i)}_{\mu_{i}}F^{\mu_{i}-1}_{i}.$
We may assume 
$$e_{Q}=\left(
\begin{array}{ccc}
E_{1}&&\text{\large{0}}\\
&\ddots&\\
\text{\large{0}}&&E_{n}\\
\end{array}
\right)
$$
and 
$$f_{Q}=\left(
\begin{array}{ccc}
F_{1}&&\text{\large{0}}\\
&\ddots&\\
\text{\large{0}}&&F_{n}\\
\end{array}
\right).
$$
Any element of $e_{Q}+Z_{\mf{l}_{Q}}(f_{Q})$ can be written as
$$Z(x):=\left(
\begin{array}{ccc}
Z_{1}(x^{(1)})&&\text{\large{0}}\\
&\ddots&\\
\text{\large{0}}&&Z_{n}(x^{(n)})\\
\end{array}
\right)
$$
for some $x^{(i)}_{j}$'s. We regard $x^{(i)}_{j}$'s as coordinates on $e_{Q}+Z_{\mf{l}_{Q}}(f_{Q})$ and define $\tilde{e}_{r}(\mu;i)\in\bb{C}[e_{Q}+Z_{\mf{l}_{Q}}(f_{Q})]$ for $1\leq r\leq\mu_{i}$ by $$\det(tI-Z_{i}(x^{(i)}))=t^{\mu_{i}}-\tilde{e}_{1}(\mu;i)t^{\mu_{i}-1}+\cdots+(-1)^{\mu_{i}}\tilde{e}_{\mu_{i}}(\mu;i).$$ Then $\tilde{e}_{r}(\mu;i)$ is homogeneous of degree $2r$ and $\{\tilde{e}_{r}(\mu;i)\mid 1\leq i\leq n,1\leq r\leq\mu_{i}\}$ freely generate the ring $\bb{C}[e_{Q}+Z_{\mf{l}_{Q}}(f_{Q})]$. We denote by $\psi:R_{\mu}\xrightarrow{\sim}\bb{C}[e_{Q}+Z_{\mf{l}_{Q}}(f_{Q})]$ the graded algebra isomorphism given by $\psi(e_{r}(\mu;i))=\tilde{e}_{r}(\mu;i)$. We denote by $\tilde{e}_{r}(\mu;i_{1},\ldots,i_{m})$ the image of $e_{r}(\mu;i_{1},\ldots,i_{m})$ under $\psi$.

In order to prove the assertion, it suffices to show that the defining ideal $\tilde{I}^{\lambda}_{\mu}$ of $\mca{N}_{P}\cap(e_{Q}+Z_{\mf{l}_{Q}}(f_{Q}))$ in $\bb{C}[e_{Q}+Z_{\mf{l}_{Q}}(f_{Q})]$ coincides with $\psi(I^{\lambda}_{\mu})$. By Theorem A.4, $\tilde{I}^{\lambda}_{\mu}$ is generated by coefficients of $t^{k}$ of various $s$-minors of $tI-Z(x)$ with $k<\lambda_{n-s+1}+\cdots+\lambda_{n}$. If we remove the first row and the last column of $Z_{i}(x^{(i)})$, then we obtain a upper triangular matrix with diagonal entries $1$. Hence for $s<\mu_{i}$, there is an $s$-minor of $tI-Z_{i}(x^{(i)})$ which equals to $\pm 1$. 

Let $l=\#\{i\mid\mu_{i}>0\}$. Consider an $s$-minor of $tI-Z(x)$. Note that nonzero $s$-minor of $tI-Z(x)$ is a certain product of $s_{i}$-minors of $tI-Z_{i}(x^{(i)})$ with $s_{1}+\cdots+s_{n}=s$. We set $m=l-(n-s)$. Then we have $\#\{i\mid s_{i}=\mu_{i}>0\}\geq m$.

First we assume $m\leq 0$. Then there is an $s$-minor of $tI-Z(x)$ which equals to $\pm 1$. If $\lambda_{n-s+1}+\cdots+\lambda_{n}=0$, then $s$-minors of $tI-Z(x)$ do not contribute to $\tilde{I}^{\lambda}_{\mu}$. If $\lambda_{n-s+1}+\cdots+\lambda_{n}\geq 1$, then we have $1\in\tilde{I}^{\lambda}_{\mu}$. On the other hand, there exists $i$ such that $\mu_{i}=0$ by the assumption $m\leq0$. We have $1=e_{0}(\mu;i)\in I^{\lambda}_{\mu}$ by the definition of $I^{\lambda}_{\mu}$ and
\begin{align*}
\mu_{i}-\lambda_{l+1}-\cdots-\lambda_{n}&=-\lambda_{n-s+m+1}-\cdots-\lambda_{n} \\
&\leq-\lambda_{n-s+1}-\cdots-\lambda_{n}\\
&<0. 
\end{align*}
Hence we have $\tilde{I}^{\lambda}_{\mu}=\psi(I^{\lambda}_{\mu})$ in this case. 

Next we consider the case of $m\geq 1$. Let $1\leq i_{1}<\cdots<i_{m}\leq l$ be some labels satisfying $s_{i}=\mu_{i}$. Then this $s$-minor of $tI-Z(x)$ is a product of some polynomial and 
\begin{align*}
\det(tI-Z_{i_{1}})\cdots\det(tI-Z_{i_{m}})&=(t^{\mu_{i_{1}}}-\tilde{e}_{1}(\mu;i_{1})t^{\mu_{i_{1}}-1}+\cdots+(-1)^{\mu_{i_{1}}}\tilde{e}_{\mu_{i_{1}}}(\mu;i_{1}))\cdot\\ &\hspace{3em}\cdots(t^{\mu_{i_{m}}}-\tilde{e}_{1}(\mu;i_{m})t^{\mu_{i_{m}}-1}+\cdots+(-1)^{\mu_{i_{m}}}\tilde{e}_{\mu_{i_{m}}}(\mu;i_{m}))\\
&=\sum_{r=0}^{\mu_{i_{1}}+\cdots+\mu_{i_{m}}}(-1)^{r}\tilde{e}_{r}(\mu;i_{1},\ldots,i_{m})t^{\mu_{i_{1}}+\cdots+\mu_{i_{m}}-r}.
\end{align*}
Hence the coefficients of $t^{k}$ with $k<\lambda_{n-s+1}+\cdots+\lambda_{n}$ are contained in the ideal generated by $\{\tilde{e}_{r}(\mu;i_{1},\ldots,i_{m})\mid r>\mu_{i_{1}}+\cdots+\mu_{i_{m}}-\lambda_{l-m+1}-\cdots-\lambda_{n}\}$. Conversely, if we choose $s_{i}=\mu_{i}-1$ for $i\neq i_{1},\ldots,i_{m}$ and $s_{i}$-minors of $tI-Z_{i}$ which are $\pm 1$, then $\pm\tilde{e}_{r}(\mu;i_{1},\ldots,i_{m})$ appears as a coefficient of $t^{k}$ for some $s$-minor of $(tI-Z(x))$ with $k<\lambda_{n-s+1}+\cdots+\lambda_{n}$. This proves $\tilde{I}^{\lambda}_{\mu}=\psi(I^{\lambda}_{\mu})$.
\end{proof}

\section{Hypertoric variety}

We briefly recall the definition and some properties of hypertoric varieties following \cite{P}. Let $T^{n}=(\bb{G}_{m})^{n}$ be the $n$-dimensional complex torus and $\mf{t}^{n}$ its Lie algebra with a full lattice $\mf{t}^{n}_{\bb{Z}}$ and its basis $\{\varepsilon_{i}\}$. Let $\mf{t}^{d}$ be a complex vector space of dimension $d$ with a full lattice $\mf{t}^{d}_{\bb{Z}}$. Let $\{a_{1},\ldots,a_{n}\}\subset\mf{t}^{d}_{\bb{Z}}$ be a collection of nonzero vectors which spans $\mf{t}^{d}_{\bb{Z}}$. Let $a:\mf{t}^{n}\rightarrow\mf{t}^{d}$ be the linear map defined by $a(\varepsilon_{i})=a_{i}$ and let $\mf{t}^{k}$ be the kernel of $a$ with a full lattice $\mf{t}^{k}_{\bb{Z}}$. Then we have the following exact sequences $$0\rightarrow\mf{t}^{k}\xrightarrow{\iota}\mf{t}^{n}\xrightarrow{a}\mf{t}^{d}\rightarrow 0$$ and $$0\rightarrow\mf{t}^{k}_{\bb{Z}}\rightarrow\mf{t}^{n}_{\bb{Z}}\rightarrow\mf{t}^{d}_{\bb{Z}}\rightarrow 0.$$ This gives an exact sequence of tori $$0\rightarrow T^{k}\rightarrow T^{n}\rightarrow T^{d}\rightarrow 0.$$ 
There is a Hamiltonian $T^{n}$ action on $T^{\ast}\bb{C}^{n}$ given by $$(\lambda_{1},\ldots,\lambda_{n})\cdot (z_{1},\ldots,z_{n},w_{1},\ldots,w_{n})=(\lambda_{1}z_{1},\ldots,\lambda_{n}z_{n},\lambda_{1}^{-1}w_{1},\ldots,\lambda_{n}^{-1}w_{n})$$ for $(\lambda_{1},\ldots,\lambda_{n})\in T^{n}$ and $(z_{1},\ldots,z_{n},w_{1},\ldots,w_{n})\in T^{\ast}\bb{C}^{n}$. The moment map $\mu_{n}:T^{\ast}\bb{C}^{n}\rightarrow(\mf{t}^{n})^{\ast}$ for this action is given by $\mu_{n}(z_{1},\ldots,z_{n},w_{1},\ldots,w_{n})=(z_{1}w_{1},\ldots,z_{n}w_{n})$. Then the moment map for the action of $T^{k}$ on $T^{\ast}\bb{C}^{n}$ is given by $\mu=\iota^{\ast}\circ\mu_{n}$. For $\alpha\in(\mf{t}^{k}_{\bb{Z}})^{\ast}$ a character of $T^{k}$, we define the hypertoric variety associated to $\mca{A}=\{a_{1},\ldots,a_{n}\}$ and $\alpha$ by $$\mf{M}_{\alpha}(\mca{A})=\mu^{-1}(0)/\!\!/_{\alpha}T^{k}.$$ Here the quotient above is the GIT quotient with respect to $\alpha$. For $r=(r_{1},\ldots,r_{n})\in(\mf{t}^{n})^{\ast}$ a lift of $\alpha$ along $\iota^{\ast}$ and $i=1,\ldots,n$, we set $$H_{i}=\{x\in(\mf{t}^{d})^{\ast}_{\bb{R}}\mid x\cdot a_{i}+r_{i}=0\}.$$ Hyperplane arrangement $\{H_{1},\ldots,H_{n}\}$ is called simple if every subset of $m$ hyperplanes with nonempty intersection intersects in codimension $m$ and $\mca{A}$ is called unimodular if every collection of $d$ linearly independent vectors $\{a_{i_{1}},\ldots,a_{i_{d}}\}$ spans $\mf{t}^{d}_{\bb{Z}}$ over $\bb{Z}$. It is known that $\mf{M}_{\alpha}(\mca{A})$ is smooth if and only if $\{H_{1},\ldots,H_{n}\}$ is simple and $\mca{A}$ is unimodular. The torus $T^{d}$ naturally acts on $\mf{M}_{\alpha}(\mca{A})$ and preserves the symplectic form.

We set $\check{\mf{t}}^{n}=(\mf{t}^{n})^{\ast}$, $\check{\mf{t}}^{d}=(\mf{t}^{d})^{\ast}$, and $\check{\mf{t}}^{k}=(\mf{t}^{k})^{\ast}$. We denote by $\check{T}^{n}$, $\check{T}^{d}$, and $\check{T}^{k}$ the dual tori of $T^{n}$, $T^{d}$, and $T^{k}$ respectively. We set $b=\iota^{\ast}$ and set $b_{i}\in\check{\mf{t}}^{k}_{\bb{Z}}$ $(1\leq i\leq n)$ to be the image of the standard basis of $\check{\mf{t}}^{n}_{\bb{Z}}$ under $b$. Let $\mca{B}=\{b_{1},\ldots,b_{n}\}$ be the Gale dual configuration of $\mca{A}$. Let us fix a basis of $\mf{t}^{d}_{\bb{Z}}\cong\bb{Z}^{d}$ and its dual basis $\check{\mf{t}}^{d}_{\bb{Z}}\cong\bb{Z}^{d}$. Let us write $a_{i}=(a_{i1},\ldots,a_{id})$ using this basis. Then the moment map $\check{\mu}:T^{\ast}\bb{C}^{n}\rightarrow\mf{t}^{d}$ for the $\check{T}^{d}$-action on $T^{\ast}\bb{C}^{n}$ is given by 
\begin{align}\label{moment}
\check{\mu}(z_{1},\ldots,z_{n},w_{1},\ldots,w_{n})=\left(\sum_{i}a_{i1}z_{i}w_{i},\ldots,\sum_{i}a_{id}z_{i}w_{i}\right).
\end{align}
The action of $(\lambda_{1},\ldots,\lambda_{d})\in\check{T}^{d}$ on $\bb{C}[T^{\ast}\bb{C}^{n}]$ is given by
\begin{align*} 
(\lambda_{1},\ldots,\lambda_{d})\cdot z_{i}&=\lambda_{1}^{a_{i1}}\cdots\lambda_{d}^{a_{id}}z_{i},\\
(\lambda_{1},\ldots,\lambda_{d})\cdot w_{i}&=\lambda_{1}^{-a_{i1}}\cdots\lambda_{d}^{-a_{id}}w_{i}.
\end{align*}

We consider the hypertoric variety $\mf{M}_{0}(\mca{B})=\mr{Spec}(\bb{C}[\check{\mu}^{-1}(0)]^{\check{T}^{d}})$ associated to $\mca{B}$ and $0\in(\mf{t}^{d})^{\ast}_{\bb{Z}}$. The torus $\check{T}^{k}=\check{T}^{n}/\check{T}^{d}$ naturally acts on $\mf{M}_{0}(\mca{B})$ and there is another $\bb{G}_{m}$ action on $\mf{M}_{0}(\mca{B})$ induced from its action on $T^{\ast}\bb{C}^{n}$ given by $t\cdot(z,w)=(t^{-1}z,t^{-1}w)$. This $\bb{G}_{m}$-action induces an $\bb{G}_{m}$-action on the fixed point scheme $\mf{M}_{0}(\mca{B})^{\check{T}^{k}}$. The aim of this appendix is to prove the following.

\begin{thm}
If $\mf{M}_{\alpha}(\mca{A})$ is smooth, then there is an isomorphism of graded algebras $$H^{\ast}(\mf{M}_{\alpha}(\mca{A}),\bb{C})\cong\bb{C}[\mf{M}_{0}(\mca{B})^{\check{T}^{k}}].$$ Here, grading on $\bb{C}[\mf{M}_{0}(\mca{B})^{\check{T}^{k}}]$ comes from the $\bb{G}_{m}$-action above.
\end{thm}

Let $\Delta_{\mca{A}}$ be the matroid complex associated to $\mca{A}$, that is, the simplicial complex consisting of all sets $S\subset\{1,\ldots,n\}$ such that $\{a_{i}\mid i\in S\}$ are linearly independent. Let $$\mca{SR}(\Delta_{\mca{A}}):=\bb{C}[e_{1},\ldots,e_{n}]/\left(\prod_{i\in S}e_{i}\mid S\notin\Delta_{\mca{A}}\right)$$ be the Stanley-Reisner ring of $\Delta_{\mca{A}}$. We define its grading by setting $\deg(e_{i})=2$. Then the cohomology ring of $\mf{M}_{\alpha}(\mca{A})$ can be described as follows.

\begin{thm}[\cite{HS},\cite{K}]
If $\mf{M}_{\alpha}(\mca{A})$ is smooth (or has at worst orbifold singularities), then there is an isomorphism of graded algebras $$H^{\ast}(\mf{M}_{\alpha}(\mca{A}),\bb{C})\cong\mca{SR}(\Delta_{\mca{A}})/\left(\sum_{i=1}^{n}a_{ij}e_{i}\mid j=1,\ldots,d\right).$$ 
\end{thm}

\begin{proof}[Proof of Theorem B.1]

For $x\in\bb{Z}$, we write $[x]_{+}:=\max(x,0)$. We set $u_{i}:=z_{i}w_{i}$ for $1\leq i\leq n$ and $v_{\vec{m}}=\prod_{i}z_{i}^{[m_{i}]_{+}}w_{i}^{[-m_{i}]_{+}}$ for $\vec{m}=(m_{1},\ldots,m_{n})\in\bb{Z}^{n}$ with $\sum_{i}m_{i}a_{i}=0$. Then we have $u_{i}$, $v_{\vec{m}}\in\bb{C}[z_{1},\ldots,z_{n},w_{1},\ldots,w_{n}]^{\check{T}^{d}}=\bb{C}[T^{\ast}\bb{C}^{n}/\check{T}^{d}]$. If a monomial $\prod_{i}z_{i}^{c_{i}}w_{i}^{c'_{i}}$ is contained in $\bb{C}[T^{\ast}\bb{C}^{n}/\check{T}^{d}]$, then we have $\sum_{i}(c_{i}-c'_{i})a_{i}=0$. Hence we can write $\prod_{i}z_{i}^{c_{i}}w_{i}^{c'_{i}}=v_{\vec{m}}\prod_{i}u_{i}^{\min(c_{i},c'_{i})}$ by setting $m_{i}=c_{i}-c'_{i}$. Therefore, $\bb{C}[T^{\ast}\bb{C}^{n}/\check{T}^{d}]$ is generated by $\{v_{\vec{m}}\mid\sum_{i}m_{i}a_{i}=0\}$ as a $\bb{C}[u_{1},\ldots,u_{n}]$-module and the defining ideal of the $\check{T}^{k}$-fixed point scheme $(T^{\ast}\bb{C}^{n}/\check{T}^{d})^{\check{T}^{k}}$ in $T^{\ast}\bb{C}^{n}/\check{T}^{d}$ is generated by $v_{\vec{m}}$'s for $\vec{m}\neq0$. It follows that $\bb{C}[(T^{\ast}\bb{C}^{n}/\check{T}^{d})^{\check{T}^{k}}]$ is generated by $u_{1},\ldots,u_{n}$ as a $\bb{C}$-algebra.

Let $S\subset\{1,\ldots,n\}$ be a circuit of $\Delta_{\mca{A}}$, i.e. minimal among the subsets of $\{1,\ldots,n\}$ which is not in $\Delta_{\mca{A}}$. There is a relation $\sum_{i\in S}p_{i}a_{i}=0$, where all $p_{i}\in\bb{Z}$ are nonzero. For any $i_{0}\in S$, $a_{i_{0}}=-\sum_{i\in S\setminus\{i_{0}\}}\frac{p_{i}}{p_{i_{0}}}a_{i}$ and $\{a_{i}\}_{i\in S\setminus\{i_{0}\}}$ is linearly independent. Hence from the unimodularity of $\mca{A}$, we have $\frac{p_{i}}{p_{i_{0}}}\in\bb{Z}$. Therefore, we can take $p_{i}=\pm1$ for all $i\in S$. We set $p_{i}=0$ for $i\notin S$ and set $\vec{p}=(p_{1},\ldots,p_{n})\in\bb{Z}^{n}$. Then we have $\prod_{i\in S}u_{i}=v_{\vec{p}}v_{-\vec{p}}$ in $\bb{C}[T^{\ast}\bb{C}^{n}/\check{T}^{d}]$ and hence $\prod_{i\in S}u_{i}=0$ in $\bb{C}[(T^{\ast}\bb{C}^{n}/\check{T}^{d})^{\check{T}^{k}}]$. On the other hand, if a monomial $\prod_{i}u_{i}^{q_{i}}$ in $u_{i}$ is zero in $\bb{C}[(T^{\ast}\bb{C}^{n}/\check{T}^{d})^{\check{T}^{k}}]$, then there exists $\vec{c},\vec{m},\vec{m}'\in\bb{Z}^{n}$ with $\sum_{i}m_{i}a_{i}=\sum_{i}m'_{i}a_{i}=0$ and $\vec{m}\neq0$ such that $$\prod_{i}u_{i}^{q_{i}}=v_{\vec{m}}v_{\vec{m}'}\prod_{i}u_{i}^{c_{i}}.$$ Hence the subset $\{i\mid q_{i}\neq0\}\subset\{1,\ldots,n\}$ contains $\{i\mid m_{i}\neq0\}\notin\Delta_{\mca{A}}$. Therefore, $\prod_{i}u_{i}^{q_{i}}$ is contained in the ideal of $\bb{C}[u_{1}\ldots,u_{n}]$ generated by $\prod_{i\in S}u_{i}$ for $S\notin\Delta_{\mca{A}}$. It follows that we have an isomorphism of graded algebras $$\bb{C}[(T^{\ast}\bb{C}^{n}/\check{T}^{d})^{\check{T}^{k}}]\cong\bb{C}[u_{1},\ldots,u_{n}]/\left(\prod_{i\in S}u_{i}\mid S\notin\Delta_{\mca{A}}\right)\cong\mca{SR}(\Delta_{\mca{A}})$$ by sending $u_{i}$ to $e_{i}$. By (\ref{moment}), we have $$\bb{C}[\mf{M}_{0}(\mca{B})]\cong\bb{C}[T^{\ast}\bb{C}^{n}/\check{T}^{d}]/\left(\sum_{i}a_{ij}u_{i}\mid j=1,\ldots,d\right).$$ It follows that 
\begin{align*}
\bb{C}[\mf{M}_{0}(\mca{B})^{\check{T}^{k}}]&\cong\bb{C}[(T^{\ast}\bb{C}^{n}/\check{T}^{d})^{\check{T}^{k}}]/\left(\sum_{i}a_{ij}u_{i}\mid j=1,\ldots,d\right)\\
&\cong\mca{SR}(\Delta_{\mca{A}})/\left(\sum_{i=1}^{n}a_{ij}e_{i}\mid j=1,\ldots,d\right).
\end{align*}
By Theorem B.2, this implies Theorem B.1.

\end{proof}


\begin{thebibliography}{ABCD}

\bibitem{BLPW}
{T. Braden, A. Licata, N. Proudfoot, B. Webster},
{Quantizations of conical symplectic resolutions II: category $\mca{O}$ and symplectic duality},
{arXiv:1407.0964}


\bibitem{BO}
{J. Brundan, V. Ostrik},
{Cohomology of {S}paltenstein varieties},
{Transform. Groups, 16 (2011), 619--648}

%\bibitem{CG}
%{N. Chriss, V. Ginzburg},
%{Representation theory and complex geometry},
%{Birkh\"auser Boston Inc., Boston, MA, 1997}


\bibitem{D}
{J. Dalbec},
{Multisymmetric functions},
{Beitr\"age Algebra Geom., 40 (1999), 27--51}


\bibitem{DP}
{C. DeConcini, C. Procesi},
{Symmetric functions, conjugacy classes and the flag variety},
{Invent. Math., 64 (1981), 203--219}

\bibitem{F}
{J. Fogarty},
{Fixed point schemes},
{Amer. J. Math., 95 (1973), 35--51}

\bibitem{G}
{I.M. Gessel},
{Enumerative applications of symmetric functions},
{Actes $17^{e}$ S\'eminaire Lotharingien, Publ. I.R.M.A. Strasbourg, 348 (1988), 5--21}

\bibitem{HS}
{T. Hausel, B. Sturmfels},
{Toric hyper{K}\"ahler varieties},
{Doc. Math., 7 (2002), 495--534}

\bibitem{K}
{H. Konno},
{Cohomology rings of toric hyperk\"ahler manifolds},
{Internat. J. Math., 11 (2000), 1001--1026}

\bibitem{LS}
{M. Lehn, C. Sorger},
{Symmetric groups and the cup product on the cohomology of {H}ilbert schemes},
{Duke Math. J., 110 (2001), 345--357}

\bibitem{N}
{H. Nakajima},
{Lectures on {H}ilbert schemes of points on surfaces},
{University Lecture Series, vol. 18, American Mathematical Society, Providence, RI, 1999}


\bibitem{T}
{T. Tanisaki},
{Defining ideals of the closures of the conjugacy classes and representations of the {W}eyl groups},
{T\^ohoku Math. J. (2), 34 (1982), 575--585}

\bibitem{P}
{N. Proudfoot},
{A survey of hypertoric geometry and topology},
{Contemp. Math., 460 (2008), 323--338}

\bibitem{V}
{E. Vasserot},
{Sur l'anneau de cohomologie du sch\'ema de {H}ilbert de {$\bold C^2$}},
{C. R. Acad. Sci. Paris S\'er. I Math., 332 (2001), 7--12}

\bibitem{W}
{J.Weyman},
{The equations of conjugacy classes of nilpotent matrices},
{Invent. Math., 98 (1989), 229--245}





\end{thebibliography}
\end{document}